\documentclass[12pt]{amsart}
\usepackage{amsmath,amsthm,amsfonts,amscd,amssymb,eucal,latexsym,mathrsfs, appendix, yhmath, setspace}
\usepackage{color,colortbl}
\usepackage[numbers,sort&compress]{natbib}
\usepackage[all,cmtip]{xy}
\usepackage{enumerate}

\newtheorem{theorem}{Theorem}[section]
\newtheorem{corollary}[theorem]{Corollary}
\newtheorem{lemma}[theorem]{Lemma}
\newtheorem{proposition}[theorem]{Proposition}

\theoremstyle{definition}
\newtheorem{definition}[theorem]{Definition}
\newtheorem{remark}[theorem]{Remark}

\newtheorem{example}[theorem]{Example}
\newtheorem{question}[theorem]{Question}

 \newcommand{\Map}{{\rm Map}}
\newcommand{\rk}{{\rm rk}}
\newcommand{\mrk}{{\rm mrk}}
\newcommand{\mrks}{{\rm mrk_\Sigma}}
\newcommand{\nv}{{\rm nv}}
\newcommand{\ocap}{{\rm ocap}}

\newcommand{\mdim}{{\rm mdim}}
\newcommand{\mdimsm}{\mdim_{\Sigma, \rM}}

\newcommand{\mdims}{{\rm mdim_{\Sigma}}}

\newcommand{\Wdim}{{\rm Wdim}}
\newcommand{\Widim}{{\rm Widim}}

\newcommand{\rM}{{\rm M}}
\newcommand{\Sym}{{\rm Sym}}

 \newcommand{\cB}{{\mathcal B}}
  
  \newcommand{\cD}{{\mathcal D}}
    \newcommand{\cE}{{\mathcal E}}
  \newcommand{\cF}{{\mathcal F}}

 \newcommand{\cM}{{\mathcal M}}

  \newcommand {\cV}{{\mathcal V}}
    \newcommand {\cW}{{\mathcal W}}
 \newcommand{\cZ}{{\mathcal Z}}

 \newcommand{\bN}{{\mathbb N}}

 \newcommand{\bR}{{\mathbb R}}
 
 \newcommand{\bZ}{{\mathbb Z}}
  
 \newcommand{\ZG}{{\mathbb Z \Gamma}}

 \newcommand{\hcM}{\widehat{\mathcal M}}

 \newcommand{\sA}{{\mathscr A}}
\newcommand{\sB}{{\mathscr B}}

\newcommand{\sF}{{\mathscr F}}

\newcommand{\sM}{{\mathscr M}}

\newcommand{\sU}{{\mathscr U}}

\allowdisplaybreaks

\begin{document}

\title{Naive mean dimension}

\author[B.~Liang]{Bingbing Liang}
\author[K.~Yan]{Kesong Yan}

\address{\hskip-\parindent
K.Y., School of Mathematics and Statistics, Hainan Normal University, Haikou, Hainan, 571158, China.}
\email{ksyan@mail.ustc.edu.cn}

\subjclass[2020]{Primary 37B02, 37A35, 22B05, 20C07.}
\keywords{naive mean dimension,  sofic mean dimension, naive mean rank, naive metric mean dimension}

\begin{abstract}
We investigate the dynamical property of the naive mean dimension for continuous actions of any countable group on compact metrizable spaces. It is shown that naive mean dimension serves as an upper bound of sofic mean dimension for actions of nonamenable groups. For algebraic actions we obtain more satisfactory results by looking at the naive version of mean rank for modules over integral group rings. We also consider the naive metric mean dimension and investigate its relations with  sofic metric mean dimension.
\end{abstract}

\maketitle


\section{Introduction}

Naive entropies  (both measure-theoretic and topological) are  natural dynamical invariants for actions of any countable group. Bowen first studied it for measure-preserving actions and showed that for many nonamenable groups, the measures of zero naive entropy are generic \cite{B16} and  sofic entropy for smooth actions of a slighter bigger group (countable group with an infinite cyclic subgroup of infinite index) is non-positive \cite{B20}. It is known that for ergodic measure-preserving actions naive entropy serves as an upper bound for Rokhlin entropy introduced by Seward \cite{B20, S19b, S19a}. Burton studied the naive topological entropy as a topological counterpart for continuous actions and proved that naive topological entropy serves as an upper bound of sofic entropy \cite{B17}. Taking advantage of the variational principle, Downarowicz, Frej, and Romagnoli proved that for continuous actions of amenable groups, naive topological entropy coincides with the classical topological entropy \cite{DFR16}. Li and Rong showed that the positivity of naive entropy provides a sufficient criterion for Li-Yorke chaos \cite{LR21}.

In this paper, we are interested in studying a naive version of another dynamical invariant---mean dimension,  called {\it naive mean dimension} (Definition \ref{nmdim}),  for continuous actions of any countable group on compact metrizable spaces. Mean dimension was introduced by Gromov and was immediately applied to the embedding problem of dynamical systems \cite{G99m, LW00}. There have been some important progress of mean dimension theory in topological dynamics \cite{L99, G15, GLT16, GQT19, GT20, CLS21, YCZ22, CDZ22}. The notion of mean dimension was extended to the actions of sofic groups \cite{Li13}. Sofic groups, introduced by Gromov, encompass a broad class of groups that includes amenable groups and residually finite groups \cite{G99s}. Further applications and connections can be found in geometric measure theory \cite{MT15, T18}, $L^2$-invariant theory \cite{LL18, LL19, L19}, classification of $C^\ast$-algebras \cite{EZ17, N21}, information theory \cite{LT18, LT19, T20}, and fractal geometry \cite{T22, LSL24, LL24}.

Our first main result compares the naive mean dimension and sofic mean dimension for actions of sofic groups, whose proof combines the method from \cite{B17} and technique from \cite{Li13}.
\begin{theorem} \label{main 1}
Let $\Gamma \curvearrowright X$ be any continuous action of a sofic group $\Gamma$ on a compact metrizable space. If the naive mean dimension of $\Gamma \curvearrowright X$ is zero, then the sofic mean dimension of $\Gamma \curvearrowright X$ is nonpositive.
\end{theorem}
In particular, since the naive mean dimension can only take values of either zero or infinity for actions of nonamenable groups (Proposition \ref{range}), we conclude that  naive mean dimension provides an upper bound for sofic mean dimension in the context of nonamenable group actions. This relationship highlights a significant distinction in comparison to actions of amenable groups, where it is clear by definition that mean dimension serves as an upper bound for naive mean dimension. It is then natural to ask
\begin{question}
    Let $\Gamma \curvearrowright X$ be any continuous action of a sofic group $\Gamma$ on a compact metrizable space. Is it true that naive mean dimension of $\Gamma \curvearrowright X$ bounds above the mean dimension of $\Gamma \curvearrowright X$?
\end{question}

For actions of full shifts we can give an affirmative answer by indirect computation (Propositions \ref{nonamenable example} and \ref{amenable example}). For {\it an algebraic action $\Gamma \curvearrowright X$}, that is, $\Gamma \curvearrowright X$ is a continuous action by group of automorphisms on a compact metrizable abelian group, we can instead take advantage of Pontryagin duality $\widehat{X}$ of $X$ to consider its algebraic counterpart $\widehat{X}$. This guides our attention to the notion of {\it naive mean rank} for modules over the integral group ring $\ZG$ as a naive version of mean rank introduced in \cite{LL18} (see Definition \ref{nmrk def}).  It turns out that naive mean rank is easier to study and we obtain a complete result on the comparison between naive mean rank and sofic mean rank, the latter of which was introduced in \cite{LL19}.
\begin{theorem} \label{nmrk compare}
    Let $\Gamma$ be a sofic group (not necessarily infinite) and $\cM$ any $\ZG$-module. Then the naive mean rank of $\cM$ bounds above the sofic mean rank of $\cM$. 
\end{theorem}
Taking advantage of a bridge between naive mean rank and naive mean dimension, we obtain:
\begin{corollary} \label{main 2}
    For any algebraic action $\Gamma \curvearrowright X$ of a sofic group (not necessarily infinite), we have naive mean dimension of $\Gamma \curvearrowright X$ bounds above the sofic mean dimension of $\Gamma \curvearrowright X$.
\end{corollary}

We also investigate the naive counterpart of metric mean dimension, referred to as {\it naive metric mean dimension}. There are two natural questions arising:
\begin{question}
    Does the naive metric mean dimension take a value of either zero or infinity for any action of nonamenable groups? Does zero naive metric mean dimension imply nonpositive sofic metric mean dimension?
\end{question}

It turns out that there are new obstructions to studying these fundamental questions. For the first question, the obstruction is mainly due to the dependence of metric mean dimension on the underlying metric. Unlike naive topological entropy, the naive metric mean dimension quantifies the decay rate of naive topological entropy at the  $\varepsilon$-scale, which we may refer to as the {\it naive $\varepsilon$-entropy}. For the second question, the new obstruction is how to control the decay rate of naive $\varepsilon$-entropy. If naive $\varepsilon$-entropy decays fast enough to zero, we can use the technique for the naive mean dimension to show sofic mean dimension is nonpositive (Theorem \ref{decaying case}). For algebraic actions, we can take advantage of naive mean rank to have a complete comparison between naive metric mean dimension and sofic metric mean dimension (Proposition \ref{algebraic action comparision}).

Since the small boundary property indicates that the systems have zero mean dimension, in Section 6, we discuss a naive version of the small boundary property that implies the systems have zero naive mean dimension.

We organize the paper as follows. In Section 2, we mainly recall the notions of amenable groups, sofic groups, and algebraic actions. In Section 3, we introduce the notion of naive mean dimension, discuss its interesting properties, and prove Theorem \ref{main 1}. In Section 4, we introduce the notion of naive mean rank and prove Theorem \ref{nmrk compare}. In Section 5, we consider the notion of naive metric mean dimension and discuss its relation with previous naive invariants. We end the paper by discussing the notion of naive small boundary property in Section 6.



\medskip
\noindent{\it Acknowledgments.}  
The first author is supported by NSFC grant 12271387. The second author is supported by NSFC grant 12171175. The first author would like to thank Professor Xiaojun Huang for a stimulating question. We are grateful for the careful reading and helpful suggestions from the referees.

\section{Preliminaries}
\numberwithin{equation}{section}
\setcounter{equation}{0}

 Throughout this paper $\Gamma$ is always a countable discrete group with identity element $e_\Gamma$. Denote by $\cF(\Gamma)$ the collection of nonempty finite subsets of $\Gamma$.

By a {\it (topological) dynamical system}, denoted by $\Gamma \curvearrowright X$, we mean a continuous action of a discrete group $\Gamma$ on a compact metrizable space $X$.  A {\it factor map} $\pi \colon X \to Y$  between two dynamical systems $\Gamma \curvearrowright X$ and $\Gamma \curvearrowright Y$, is a surjective continuous map such that $\pi(sx)=s\pi(x)$ for every $s \in \Gamma$ and $x \in X$.  
\subsection{Amenable and sofic groups} The group $\Gamma$ is called {\it amenable} if for every $K\in \cF(\Gamma)$ and every $\delta>0$ there exists an $F\in \cF(\Gamma)$ with $|KF\setminus F|<\delta |F|$.  


For a $\bR$-valued function $\varphi$ defined on $\cF(\Gamma)$, we say that   {\it $\varphi(F)$ converges to $c\in \bR$ when $F\in \cF(\Gamma)$ becomes more and more left
invariant}, denoted by 
$$\lim_F\varphi(F)=c,$$
if for any $\varepsilon>0$ there are some $K \in \cF(\Gamma)$ and $\delta > 0$ such that  $|\varphi(F)-c|<\varepsilon$ for all $F \in \cF(\Gamma)$ satisfying $|KF\setminus F|<\delta |F|$.

The following strong version of Ornstein-Weiss lemma is mainly due to Gromov \cite[1.3.1]{G99m}\cite[Theorem 4.38, Theorem 4.48]{KLb} \cite[Theorem 9.4.1]{C15}.
\begin{lemma} \label{OW}
Let  $\varphi \colon \cF(\Gamma)  \to [0, +\infty)$ be a function satisfying\\
\begin{enumerate}
    \item $\varphi(Fs)=\varphi(F)$ for all $F \in \cF(\Gamma)$ and $s \in \Gamma$; \\
    \item $\varphi(F_1\cup F_2) \leq \varphi(F_1) + \varphi(F_2)$ for all $F_1, F_2 \in \cF(\Gamma)$.\\
\end{enumerate}
Then the limit $\lim_F\varphi(F)/|F|$ exists. Moreover if $\varphi$ is strongly subadditive in the sense that
$$\varphi(F_1\cup F_2) \leq \varphi(F_1) + \varphi(F_2)-\varphi(F_1\cap F_2)$$
for all $F_1, F_2 \in \cF(\Gamma)$, we have
$$\lim_F\frac{\varphi(F)}{|F|}=\inf_{F \in \cF(\Gamma)} \frac{\varphi(F)}{|F|}.$$
\end{lemma}

For every $d \in \bN$ write $[d]$ for the set $\{1, 2, \dots, d\}$ and $\Sym(d)$ for the symmetric group over $[d]$. Given a map $\sigma \colon \Gamma \to \Sym(d)$ we adopt the convention $\sigma(s)(v)=\sigma_s(v)$ for every $s \in \Gamma$ and $v \in [d]$.  
\begin{definition}
A sequence of maps $\Sigma=\{\sigma_i: \Gamma \to {\rm \Sym} (d_i)\}_{i \in \bN}$  is called a {\it sofic approximation} for $\Gamma$  if it satisfies:
\begin{enumerate}
\item $\lim_{i\to \infty}|\{v\in [d_i]: \sigma_{i,s}\sigma_{i,t}(v)=\sigma_{i, st}(v)\}|/d_i=1$ for all $s, t\in \Gamma$,

\item $\lim_{i\to \infty}|\{v\in [d_i]: \sigma_{i, s}(v)\neq \sigma_{i,t}(v)\}|/d_i=1$ for all distinct $s, t\in \Gamma$,

\item $\lim_{i\to \infty} d_i=+\infty$.
\end{enumerate}
The group $\Gamma$ is called {\it sofic} if it admits a sofic approximation. 
\end{definition}
We shall frequently say that $\sigma$ is  {\it good enough } for $\Gamma$ in the sense that for some small $\tau > 0$  and large $F \in \cF(\Gamma)$ (to be specified in the context), the inequality $|\{v \in [d]: \sigma_s\sigma_t(v)=\sigma_{st}(v)\}|/d > 1-\tau$ holds for every $s, t \in F$, and inequality $|\{v \in [d]: \sigma_s(v) \neq \sigma_t(v)\}| > 1-\tau$ holds for every distinct $s, t \in F$.

Every amenable group is sofic since one can use a sequence of asymptotically invariant subsets of the amenable group to construct a sofic approximation. Residually finite groups are also sofic since a sequence of exhausting finite-index subgroups naturally induces a sofic approximation in which case each approximating map is a genuine group homomorphism. We refer the reader to \cite{CC10, CLb} for more information on sofic groups.

\begin{definition}
    Consider a continuous action $\Gamma \curvearrowright X$ and a continuous pseudometric $\rho$ on $X$. Given $F \in \cF(\Gamma), \delta > 0$ and a map $\sigma \colon \Gamma \to \Sym(d)$ for some $d \in \bN$, we set  
$$\Map(\rho, F, \delta, \sigma)= \{\varphi \colon [d] \to X:  \rho_2(s\varphi, \varphi \circ \sigma(s)) \leq \delta {\rm  \ for \ every \ } s \in F   \},$$
where  the pseudometric $\rho_2$ on $X^{[d]}$  is defined via
$$\rho_2(\varphi, \psi)=\left(\frac{1}{d}\sum_{v \in [d]} \rho(\varphi(v), \psi(v))^2\right)^{1/2}.$$
\end{definition}
Note that $\Map(\rho, F, \delta, \sigma)$ is a closed subset of $X^{[d]}$ and its elements can be interpreted as approximately equivariant maps from $[d]$ to $X$. Meanwhile, the pseudometric $\rho_{X, \infty}$ on $X^{[d]}$ is defined via
$$\rho_\infty(\varphi, \psi)=\max_{v \in [d]} \rho(\varphi(v), \psi(v)).$$

The following estimation will be used twice in the sequel \cite[Lemma 2.8]{Li13}.
\begin{lemma} \label{Map lowerbound}
  Let $\rho$ be a continuous pseudometric on $X$, $F \in \cF(\Gamma)$ and $\delta >0$. Consider any map $\sigma \colon \Gamma \to \Sym(d)$. Then for each $\varphi \in \Map(\rho, F, \delta, \sigma)$, we have
  $$|\{v \in [d]: \rho(s\varphi(v), \varphi(\sigma_s(v))) \leq \sqrt{\delta}\} |\geq (1-\delta)d.$$
\end{lemma}

\subsection{Group rings and algebraic actions}
Let $R$ be a unital ring. The {\it group ring of $\Gamma$ with coefficients in $R$}, denoted by $R \Gamma$, is an extended ring of $R$, consisting of all finitely supported functions $f  \colon \Gamma \to R$. Conventionally, we shall write $f$ as $\sum_{s \in \Gamma} f_s s$, where $f_s \in R$ for all $s \in \Gamma$ and $f_s=0$ for all except finitely many $s \in \Gamma$. 
The addition and multiplication operations on $R \Gamma$ are defined by
$$ \sum_{s\in \Gamma}f_ss+\sum_{s\in \Gamma}g_ss=\sum_{s\in \Gamma}(f_s+g_s)s, \mbox{ and } \big(\sum_{s\in \Gamma}f_s s\big)\big(\sum_{t\in \Gamma}g_tt\big)=\sum_{s, t\in \Gamma}f_sg_t(st).$$
We shall mainly consider the case $R=\bZ$. Note that any $\ZG$-module  $\cM$ is simply an abelian group carrying a $\Gamma$-action by group automorphisms.

Recall that for every locally compact Hausdorff abelian group $G$, the {\it Pontryagin dual} $\widehat{G}$ of $G$ consists of all continuous group homomorphisms from $G$ to the unit circle $\bR/\bZ$. Under the operation of pointwise multiplication and the compact-open topology $\widehat{G}$ is again a locally compact Hausdorff group. By the classical Pontryagin duality theory, $\widehat{G}$ is metrizable if and only if $G$ is countable. In particular, for any countable $\ZG$-module $\cM$, treated as a discrete abelian group, its Pontryagin dual $\hcM$ is a compact metrizable abelian group.  Moreover, the algebraic $\Gamma$-action on $\cM$ naturally induces a topological dynamical system $\Gamma \curvearrowright \hcM$ by continuous group automorphisms defined via 
$$s\chi(u): = \chi(s^{-1}u)$$
for all $\chi \in \hcM, u \in \cM$, and $s \in \Gamma$.
Conversely, by Pontryagin duality, each continuous action of $\Gamma$ on a compact metrizable abelian group by group automorphisms arise this way and thus we call such a dynamical system an  {\it algebraic action} \cite{Sch95}.

\section{Naive mean dimension}
\numberwithin{equation}{section}
\setcounter{equation}{0}

In this section, we introduce the notion of naive mean dimension for actions of any countable discrete group, discuss its basic properties, and establish its relation with sofic mean dimension. 

Let $X, P$ denote two compact metrizable spaces and $\rho$ a continuous pseudometric on $X$. Fix $\varepsilon > 0$. Recall that a continuous map $f \colon X \to P$ is an {\it $(\varepsilon, \rho)$-embedding} if for every $x, x' \in X$ with $f(x)=f(x')$, one has $\rho(x, x') < \varepsilon$. For a closed subset $Z$ of $X$ denote by $\Wdim_\varepsilon(Z, \rho)$ the minimal (covering) dimension $\dim (P)$ of a compact metrizable space $P$ that admits an $(\varepsilon, \rho)$-embedding from $Z$ to $P$ (see \cite{C15} for more details on covering dimension). 
For each $F \in \cF(\Gamma)$ it induces a larger compatible metric $\rho_F$ on $X$ via
$$\rho_F(x, x'):=\max_{s \in F} \rho(sx, sx').$$

\begin{definition} \label{nmdim}
Let $\Gamma \curvearrowright X$ be a continuous action of any countable group $\Gamma$ on a compact metrizable space $X$. Set
$$\mdim_\varepsilon^\nv(\Gamma \curvearrowright X, \rho)=\inf_{F \in \cF(\Gamma)} \frac{\Wdim_\varepsilon(X, \rho_F)}{|F|}.$$
When $\rho$ is a compatible metric we define the {\it naive mean (topological) dimension of $\Gamma \curvearrowright X$} as 
$$\mdim^\nv(\Gamma \curvearrowright X):=\sup_{\varepsilon > 0} \mdim_\varepsilon^\nv(\Gamma \curvearrowright X, \rho).$$
\end{definition}

Clearly $\mdim^\nv(\Gamma \curvearrowright X)$ is independent of the choice of compatible metric $\rho$.  If the space $X$ and the acting group $\Gamma$ is understood in the context, we may write $\mdim_\varepsilon^\nv(\rho)$ and $\mdim^\nv(X)$ for $\mdim_\varepsilon^\nv( \Gamma \curvearrowright X, \rho)$ and $\mdim^\nv(\Gamma \curvearrowright X)$ respectively in short. 

\begin{remark} \label{Hausdorff case}
    Using the approach of Lindenstrasuss and Weiss \cite[Definition 2.6]{LW00}, one can also generalize the above definition to the setting of any continuous action of any discrete group on a compact Hausdorff space via finite open covers.
\end{remark}

As $\Gamma$ is amenable, we can apply Lemma \ref{OW} to the function $F \mapsto \Wdim_\varepsilon(X, \rho_F)$ to get a well-defined notation:
$$ \mdim_\varepsilon(\rho):= \lim_F \frac{\Wdim_\varepsilon(X, \rho_F)}{|F|}.$$
\begin{definition}
    For a  continuous action $\Gamma \curvearrowright X$ of an amenable group, the {\it mean (topological) dimension} of $\Gamma \curvearrowright X$ is defined as 
$$\mdim(\Gamma \curvearrowright X):=\sup_{\varepsilon > 0} \mdim_\varepsilon(\rho).$$
\end{definition}

Using the similar argument of \cite[Theorem 6]{B16}, we have
\begin{proposition} \label{range}
    If $\Gamma$ is nonamenable, then for any continuous action $\Gamma \curvearrowright X$, one has $\mdim^\nv(\Gamma \curvearrowright X) \in \{0, +\infty\}$.
\end{proposition}

\begin{proof}
    Suppose that there exists $\varepsilon_0 > 0$  such that $\mdim_{\varepsilon_0}^\nv(\rho):=c > 0$. Fix any $\rM > 0$. Since $\Gamma$ is not amenable, by \cite[Theorem 4.9.2]{CC10}, there exists $S \in \cF(\Gamma)$ such that $\frac{|SF|}{|F|} \geq M/c$ for every $F \in \cF(\Gamma)$. Then 
    \begin{align*}
        \mdim_{\varepsilon_0}^\nv(\rho_S) &=\inf_{F \in \cF(\Gamma)} \frac{\Wdim_{\varepsilon_0}(X, \rho_{SF})}{|F|} \\
        &=\inf_{F \in \cF(\Gamma)} \frac{|SF|}{|F|} \frac{\Wdim_{\varepsilon_0}(X, \rho_{SF})}{|SF|} \\
        &\geq \inf_{F \in \cF(\Gamma)} \frac{|SF|}{|F|} c=\rM.
    \end{align*}
Thus we obtain
$$\mdim^\nv(\Gamma \curvearrowright X) \geq \mdim_{\varepsilon_0}^\nv(\rho_S) \geq \rM.$$
\end{proof}

Observe that if a countable group $\Gamma$ admits an element $s$ with infinite order, then $\mdim(\bZ \curvearrowright^s X) \geq \mdim^\nv(\Gamma \curvearrowright X)$. Moreover, we can easily discuss the relation between naive mean dimension of subgroup actions as in \cite[Lemma 3.13]{B20}. 

\begin{proposition} \label{subgroup discussion}
      Let $\Gamma$ be a group with an infinite index subgroup $H$ and $\Gamma \curvearrowright X$ any continuous action on a compact metrizable space.
Let $H \curvearrowright X$ be the action from the restriction of  $\Gamma \curvearrowright X$.
Then we have the following.
\begin{itemize}
    \item[(1)] If $\mdim^\nv(H \curvearrowright X) < \infty$, then $\mdim^\nv(\Gamma \curvearrowright X)=0$;\\
    \item[(2)] If $H$ is amenable and $\mdim(H \curvearrowright X) < \infty$, then $\mdim^\nv(\Gamma \curvearrowright X)=0$.
\end{itemize}
\end{proposition}

\begin{proof}
 We only prove the case (1) provided that the case (2) can be proved in the same way.  Fix a compatible metric $\rho$ on $X$ and $\varepsilon > 0$. Choose $K \in  \cF(\Gamma)$ such that $sH\cap s'H=\emptyset$ for every distinct $s, s' \in K$. Pick a sequence  $F_n \in  \cF(H)$ such that 
 $$\mdim_\varepsilon^\nv(H\curvearrowright X, \rho_K)=\lim_{n \to \infty} \frac{\Wdim_\varepsilon(X, \rho_{KF_n})}{|F_n|}.$$
Thus

\begin{align*}
    \mdim_\varepsilon^\nv( \Gamma \curvearrowright X, \rho) &\leq \lim_{n \to \infty} \frac{\Wdim_\varepsilon(X, \rho_{KF_n})}{|KF_n|}\\
    &=\frac{1}{|K|}\mdim_\varepsilon^\nv(H\curvearrowright X, \rho_K)\\
    &\leq \frac{1}{|K|} \mdim^\nv(H \curvearrowright X).
\end{align*}
It concludes that 
$$\mdim^\nv(\Gamma \curvearrowright X) \leq \frac{1}{|K|} \mdim^\nv(H \curvearrowright X).$$
Since $H$ is of infinite index, we can let $|K|$ tend to infinity and the desired equality follows.
\end{proof}

\begin{proposition}
    If $H \leq \Gamma$ is a subgroup of finite index, then $$\mdim^\nv(H\curvearrowright X)=[\Gamma: H]\mdim^\nv(\Gamma \curvearrowright X).$$
\end{proposition} 

\begin{proof}
    From the proof of Proposition \ref{subgroup discussion}, it suffices to prove
    $$\mdim^\nv(\Gamma \curvearrowright X)\geq \frac{1}{[\Gamma: H]}\mdim^\nv(H \curvearrowright X).$$

Fix $\rho$ as a compatible metric on $X$. From the right coset decomposition of groups, for any finite subset $F \subseteq \Gamma$, we have
$$|F|\leq [\Gamma : H]\max_{s \in \Gamma} |Hs\cap F| =[\Gamma: H] |Hs_0\cap F|$$
for some $s_0 \in \Gamma$. It follows that for any $\varepsilon > 0$, we have
\begin{align*}
    \frac{\Wdim_\varepsilon(X, \rho_F)}{|F|} & \geq \frac{\Wdim_\varepsilon(X, \rho_{Hs_0\cap F})}{[\Gamma: H]|F \cap Hs_0|}=\frac{\Wdim_\varepsilon(X, \rho_{H\cap Fs_0^{-1}})}{[\Gamma:H]|H\cap Fs_0^{-1}|}\\
    & \geq \frac{1}{[\Gamma: H]} \inf_{K \in \cF(H)} \frac{\Wdim_\varepsilon(X, \rho_K)}{|K|}=\frac{1}{[\Gamma: H]} \mdim_\varepsilon^\nv(H \curvearrowright X).
\end{align*}
Thus $\mdim_\varepsilon^\nv(\Gamma \curvearrowright X) \geq \frac{1}{[\Gamma: H]} \mdim_\varepsilon^\nv(H \curvearrowright X)$ and the conclusion holds.
    
\end{proof}

Recall that the {\it stable dimension} of a compact metrizable space $Z$ is defined as
$${\rm stabdim}(Z):=\inf_{n \in \bN} \frac{\dim (Z^n)}{n}.$$

The continuous action  {\it full (left) shift} $\Gamma \curvearrowright Z^\Gamma$ is defined by 
$$(sx)_t:=x_{s^{-1}t}$$
for every $x=(x_t)_{t \in \Gamma}$ and $s \in \Gamma$.
\begin{proposition} \label{nonamenable example}
Let $K$ be a compact metrizable space with $0 < \dim (K) < \infty$. Suppose that $\dim (K) > 1$ or $K$ is of basic type, i.e. $\dim(K^2)=2\dim(K)$, that is, stable dimension of $K$ is positive. Then for any nonamenable group $\Gamma$ and the associated full shift $\Gamma \curvearrowright K^\Gamma$, one has 
$$\mdim^\nv(\Gamma \curvearrowright K^\Gamma)=\infty.$$
\end{proposition}

\begin{proof}
   By Proposition \ref{range} it suffices to show $\mdim^\nv(\Gamma \curvearrowright K^\Gamma)> 0$. Fix a compatible metric $\rho$ on $K$ and an enumeration of the elements of $\Gamma$ as $s_1=e_\Gamma, s_2, \dots$. It induces a compatible metric $\widetilde{\rho}$ on $K^\Gamma$ by
     \begin{equation} \label{compatible induced}
      \widetilde{\rho}(x, y)=\sum_{n=1}^\infty \frac{1}{2^n} \rho(x_{s_n}, y_{s_n}).
 \end{equation}

Denote by $\Widim_\varepsilon(\rho)$ the minimal (covering) dimension $\dim (P)$ of a polyhedron $P$ that admits an $(\varepsilon, \rho)$-embedding from $X$ to $P$.  It is an exercise that naive mean dimension can be equivalently defined via the parameter $\Widim_\varepsilon(\rho)$.
   Since $K$ is of finite dimension,  by \cite[Lemma 3.1]{T19}, there exists $\varepsilon_0 > 0$ such that for every $F \in \cF(\Gamma)$, one has\footnote{Write the notation in Lemma \cite[Lemma 3.1]{T19} as $\Wdim'_\varepsilon(X, \rho)$. It is readily verified that $\Wdim'_\varepsilon(X, \rho)=\Widim_\varepsilon(X, \rho)$.}
    $$\Widim_{\varepsilon_0}(K^F, \rho_F) \geq |F|(\dim (K)-1).$$
    It follows that for any $0 < \varepsilon < \varepsilon_0$, one has
    $$\Widim_\varepsilon(K^\Gamma, \widetilde{\rho}_F) \geq \Widim_\varepsilon(K^F, \rho_F) \geq |F|(\dim (K)-1)$$
    and hence 
 \begin{equation} \label{lower bound}
     \mdim^\nv(\Gamma \curvearrowright K^\Gamma) \geq \dim(K)-1.
 \end{equation}   
 Thus if $\dim (K) >1$, we obtain $\mdim^\nv(\Gamma \curvearrowright K^\Gamma)> 0$.

   Now we assume that $K$ is of basic type.  By the definition of naive mean dimension, for any dynamical system $\Gamma \curvearrowright X$ and the induced dynamical system $\Gamma \curvearrowright X^n$ by diagonal actions, we have
   $$\mdim^\nv(X) \leq \mdim^\nv(X^n)\leq n\mdim^\nv(X).$$
   Since $\Gamma$ is nonamenable, by Proposition \ref{range}, we obtain $\mdim^\nv(X)=\mdim^\nv(X^n)$ for every $n \in \bN$.
In particular, we have 
$$\mdim^\nv(K^\Gamma)=\mdim^\nv((K^2)^\Gamma) \geq \dim(K^2)-1 > 0.$$
   
\end{proof}


\begin{proposition} \label{amenable example}
    As $\Gamma$ is infinitely amenable, for any compact metrizable space $K$ of finite dimension, we have
    $$\mdim^\nv(K^\Gamma)={\rm stabdim}(K)=\mdim(K^\Gamma).$$
\end{proposition}

\begin{proof}
    From estimation (\ref{lower bound}), for every $n \in \bN$ we have
    $$\mdim^\nv((K^n)^\Gamma) \geq \dim (K^n)-1.$$
    Note that the full shift $(K^n)^\Gamma$ is isomorphic to the product system $(K^\Gamma)^n$. It follows that
    $n\mdim^\nv(K^\Gamma) \geq \mdim^\nv((K^\Gamma)^n)\geq \dim (K^n)-1$. Thus
    $$\mdim(K^\Gamma)\geq \mdim^\nv(K^\Gamma) \geq {\rm stabdim}(K).$$
Since $\Gamma$ is infinite,  from \cite[Corollary 10.5.5]{C15}, we have
$$\mdim^\nv(K^\Gamma) \leq \mdim(K^\Gamma) \leq {\rm stabdim}(K).$$
\end{proof}

\begin{remark} \label{pseudo example} 
   A notable difference between naive mean dimension and mean dimension is that the computation of naive mean dimension cannot be simplified to the dynamically generating pseudometric framework. Recall that a continuous pseudometric $\rho$ on a compact metrizable space $X$ carrying a  continuous action $\Gamma \curvearrowright X$ is called {\it dynamically generating} if 
   $$\sup_{s \in \Gamma}\rho(sx, sy) >0$$
   for any distinct element $x, y \in X$.
   Consider the dynamically generating pseudometric $\rho$ on $[0, 1]^\Gamma$ by $\rho(x, y)=|x_e-y_e|$. Then for every $0 < \varepsilon < 1$ it is clear  that for the full shift action $\Gamma \curvearrowright [0, 1]^\Gamma$, one has
    $$\mdim_\varepsilon^\nv(\Gamma \curvearrowright[0,1]^\Gamma, \rho) \leq 1$$ 
    and hence $\sup_{\varepsilon > 0} \mdim_\varepsilon^\nv(\Gamma \curvearrowright [0,1]^\Gamma, \rho) \leq 1$. However, by Proposition \ref{nonamenable example}, whenever $\Gamma$ is nonamenable, we have $\mdim^\nv(\Gamma \curvearrowright [0, 1]^\Gamma)=\infty$.
\end{remark}


Now we recall the concepts of sofic mean dimension and establish its relationship  with naive mean dimension. 
Throughout the rest of this section, $\Gamma$ denotes a sofic group with a fixed sofic approximation $\Sigma=\{\sigma_i \colon \Gamma \to \Sym(d_i)\}_{i=1}^\infty$.

\begin{definition}\cite[Definition 2.4]{Li13} \label{top def of mdim}
 For every $\varepsilon> 0, F \in \cF(\Gamma)$, set
$$\mdim_\Sigma^\varepsilon(\rho, F, \delta):=\varlimsup_{i\to \infty}\frac{1}{d_i}\Wdim_\varepsilon(\Map(\rho, F, \delta, \sigma_i), \rho_\infty).$$
If  $\Map(\rho, F, \delta, \sigma_i)$ is empty for all sufficiently large $i$, we set 
$\mdim_\Sigma^\varepsilon(\rho, F, \delta)=-\infty.$
Write
$$\mdim_\Sigma^\varepsilon(\rho):=\inf_{F \in \cF(\Gamma)} \inf_{\delta > 0}  \mdim_\Sigma^\varepsilon(\rho, F, \delta).$$
When $\rho$ is a compatible metric, the  {\it sofic mean (topological) dimension} of $\Gamma \curvearrowright X$ is defined as 
$$\mdim_\Sigma(\Gamma \curvearrowright X):=\sup_{\varepsilon > 0} \mdim_\Sigma^\varepsilon(\rho).$$
\end{definition}
By compactness of $X$ the sofic mean dimension is independent of the choice of compatible metric $\rho$.  

The following result provides a complete computation of the sofic mean dimension for full shifts, a topic initially explored in \cite[Theorem 7.1]{Li13}.
\begin{theorem}\cite[Theorem 1.1]{JQ21}
    Let $K$ be  a compact metrizable space of finite dimension and $\Gamma$ a sofic group with a sofic approximation $\Sigma$. Then for the associated full shift $\Gamma \curvearrowright K^\Gamma$ one has
    $\mdims(K^\Gamma)={\rm stabdim}(K)$.
\end{theorem}
In particular, based on Proposition \ref{nonamenable example}, we observe that naive mean dimension may be strictly greater than sofic mean dimension. 


The following lemma is summarized from the proof of \cite[Theorem 1.1]{B17}. We remark that the lemma is interesting as $|F|$ is far greater than $1/\tau$.  For every $F \in \cF(\Gamma)$ denote by $F^{-1}$ for the set $\{s^{-1}:  s\in F\}$.
\begin{lemma} \label{assigning}
    Let $\Gamma$ be any countable group. Fix $ 0 < \tau < 1$, $0 \leq \eta < 1$ and $F \in \cF(\Gamma)$. Suppose that a map $\sigma  \colon \Gamma \to \Sym(d)$ contains a set $\cB \subseteq [d]$ satisfying $|\cB| \geq (1-\frac{\tau}{2(|F|+1)})d$ and 
    $$\sigma_s(v)\neq \sigma_t(v), \ \sigma_{s^{-1}}(v)=(\sigma_s)^{-1}(v)$$
    for all $v \in \cB$ and $s\neq t \in F\cup F^{-1}$. Then for every set $\cW \subseteq [d]$ with $|\cW|\geq (1-\frac{\eta}{|F|+1})d$, there exist $\ell \in \bN, F_1, \dots, F_\ell \subseteq F$ and $c_1, \dots, c_\ell \in \cW$ such that
    \begin{itemize}
\item[i)] for every $k =1, \dots, \ell$, one has $|F_k|\geq \tau|F|/2 $;\\
\item[ii)] the sets $\{\sigma(F_k)c_k\}_{k=1}^\ell$ are pairwise disjoint with $|\bigsqcup_{k=1}^\ell \sigma(F_k)c_k| \geq (1-\tau-\eta)d$.
\end{itemize}
\end{lemma}

\begin{proof}
    We find such sets by induction. Put $F_1=F$ and assign any element of $\cB\cap \cW$ to be $c_1$. Inductively we obtain a maximal collection $\{\sigma(F_k)c_k\}_{k=1}^\ell$ of pairwise disjoint subsets of $[d]$ where each $F_k \subseteq F, c_k \in \cB\cap \cW$ and $|F_k| \geq \tau|F|/2$.   Put $\cV=\bigsqcup_{k=1}^\ell \sigma(F_k)c_k$. We shall show $|\cV| \geq (1-\tau -\eta)d$.

By maximality we conclude that for any $c \in \cB\cap \cW$, one has
\begin{equation} \label{key observation}
    |\sigma(F)c\cap \cV| > (1-\tau/2)|F|.
\end{equation}
Put $\cB'=(\cB\cap \cW) \cap_{s \in F} (\sigma_{s^{-1}})^{-1}(\cB\cap \cW)$. By the assumption on $\cB$ we have $|\cB'| \geq (1-\tau/2-\eta)d$.

{\bf Claim:} $|\cB'\setminus \cV| \leq \tau|\cB\cap \cW|/2$.

Therefore, we conclude that
$$|[d]\setminus \cV| \leq |[d]\setminus \cB'| +|\cB'\setminus \cV| \leq (\tau/2 +\eta) d +\tau/2 d=(\tau +\eta) d.$$

{\it Proof of Claim.}
For every $v \in \cB'$ consider the map $f_v \colon F \to \{c \in \cB\cap \cW: v \in \sigma(F)c\}$ sending $s$ to $(\sigma_s)^{-1}v$. Using the assumption on $\cB$ it is readily checked that $f_v$ is well-defined and bijective. It follows that $|F|=|\{c \in \cB\cap \cW: v \in \sigma(F)c\}|$. 

Assume the claim does not hold. Then
\begin{align*}
    \sum_{c \in \cB\cap \cW} \sum_{v \in \cB'\setminus \cV} 1_{\sigma(F)c}(v)&=\sum_{v \in \cB'\setminus \cV} \sum_{c \in \cB\cap \cW}  1_{\sigma(F)c}(v)\\
    &=\sum_{v \in \cB'\setminus \cV} |\{c \in \cB\cap \cW: v \in \sigma(F)c\}|= \sum_{v \in \cB'\setminus \cV} |F|\\
    &=|\cB'\setminus \cV||F| > \tau|\cB\cap \cW||F|/2.
\end{align*}
It implies that there exists $c_0 \in \cB\cap \cW $ such that $\sum_{v \in \cB'\setminus \cV} 1_{\sigma(F)c_0}(v) > \tau|F|/2$
and hence $|\sigma(F)c_0\setminus \cV| > \tau|F|/2$. This contradicts the inequality (\ref{key observation}).  
\end{proof}

Using the methods in the proofs of \cite[Theorem 1.1]{B17} and \cite[Theorem 3.1]{Li13}, we now prove Theorem \ref{main 1} as an analog of \cite[Theorem 1.1]{B17} for mean dimension.  
\begin{proof}[Proof of Theorem \ref{main 1}]
Fix $1> \beta > 0$ and a compatible metric $\rho$ on $X$. We shall prove $\mdim_\Sigma^\varepsilon(\rho) \leq  4\beta$ for every $\varepsilon > 0$. 

Let $\sU$  be a finite open cover of $X$ such that ${\rm diam}(U, \rho) < \varepsilon$ for each $U \in \sU$. 
Take $\tau > 0$ with  $\tau\Wdim_\varepsilon(X, \rho) \leq \beta$. By the definition of naive mean dimension there exists $F \in \cF(\Gamma)$ such that 
$$\Wdim_{\varepsilon/3}(X, \rho_F) \leq \tau\beta |F|.$$
Applying Lemma \ref{assigning} to $\eta=0$, as $\sigma \colon \Gamma \to \Sym(d)$ is good enough admitting a set $\cB \subseteq [d]$ satisfying $|\cB| \geq (1-\frac{\tau}{2(|F|+1)})$ and 
    $$\sigma_s(v)\neq \sigma_t(v), \ \sigma_{s^{-1}}(v)=(\sigma_s)^{-1}(v)$$
for all $v \in \cB$ and $s\neq t \in F\cup F^{-1}$, there exist $\ell \in \bN$, $F_1, \dots, F_\ell \subseteq F$, $c_1, \dots, c_\ell \in \cB$ such that
    \begin{itemize}
\item[i)] for every $k =1, \dots, \ell$, one has $|F_k|\geq \tau|F|/2 $;\\
\item[ii)] the sets $\{\sigma(F_k)c_k\}_{k=1}^\ell$ are pairwise disjoint with $|\bigsqcup_{k=1}^\ell \sigma(F_k)c_k| \geq (1-\tau)d$.
\end{itemize}

Observe that for any $F' \subseteq F$ with $|F'| \geq \tau|F|/2$, one has
$$\Wdim_{\varepsilon/3}(X, \rho_{F'}) \leq \Wdim_{\varepsilon/3}(X, \rho_F) \leq \tau\beta|F| \leq 2\beta|F'|.$$
In particular, for every $1\leq k \leq \ell$, we have
\begin{equation} \label{Fk esitmate}
    \Wdim_{\varepsilon/3}(X, \rho_{F_k}) \leq 2\beta|F_k|.
\end{equation}
Choose $\kappa > 0$ such that for every $x, x' \in X$ with $\rho(x, x') \leq \kappa$ one has  
  $$\rho_F(x, x') < \varepsilon/3.$$
Take $\delta > 0$ such that $\delta < \kappa^2$ and $\delta |\sU||F| < \beta$. We are left to show that as $\sigma$ is  good enough as in Lemma \ref{assigning}, 
$$\Wdim_\varepsilon(\Map(\rho, F^{-1}, \delta, \sigma), \rho_\infty) \leq  4\beta d.$$

Setting $\cZ=[d] \setminus \bigsqcup_{k=1}^\ell \sigma(F_k)c_k$, it follows that  $|\cZ| \leq \tau d$. For every $\varphi \in \Map(\rho, F^{-1}, \delta, \sigma)$, by Lemma \ref{Map lowerbound}, we have
$$|\{v \in [d]: \rho(s\varphi(v), \varphi(\sigma_s(v))) \leq \sqrt{\delta} {\rm \ for \ every \ } s \in F^{-1}\}| \geq (1-|F|\delta)d.$$
 Take a partition of unity $\{\zeta_U \}_{U \in \sU}$ of $X$ subordinate to $\sU$.
It follows that the map $ \overrightarrow{\zeta} \colon X \to [0, 1]^\sU$ sending $x$ to $(\zeta_U(x))_{U \in \sU}$ is an $(\varepsilon, \rho)$-embedding. 

Now  consider a continuous map $\overrightarrow{h} \colon \Map(\rho, F^{-1}, \delta, \sigma) \to ([0, 1]^\sU)^d$ defined by
$$(\overrightarrow{h}(\varphi))_v=\overrightarrow{\zeta}(\varphi(v))\max_{s \in F^{-1}} \left(\max \left(\rho(s\varphi(v), \varphi(\sigma_s(v)))-\kappa, 0\right)\right)$$
for every $v \in [d]$ and $\varphi \in \Map(\rho, F^{-1}, \delta, \sigma)$. 
Then the image of $\overrightarrow{h}$ sits inside the closed set
$$Z:=\{\omega \colon [d] \to [0, 1]^\sU:  |\{v \in [d]: \omega(v)=\overrightarrow{0}\}| \geq (1-|F|\delta)d   \}.$$
Since $\delta |\sU||F| < \beta$, by \cite[Corollary 1.2.6]{C15}, we have
$\dim (Z) \leq |\sU| |F|\delta d < \beta d$.
For each $k=1, \dots, \ell$,  choose an $(\varepsilon/3, \rho_{F_k})$-embedding $f_k \colon X \to P_k$ and an $(\varepsilon, \rho)$-embedding  $g \colon X \to Q$ satisfying that
\begin{equation}\label{WdimFk}
\dim (P_k)=\Wdim_{\varepsilon/3}(X, \rho_{F_k})  \rm{\ \ and \ \ } \dim (Q)=\Wdim_\varepsilon(X, \rho).
\end{equation}

Next we obtain a continuous map $\Phi \colon \Map(\rho, F^{-1}, \delta, \sigma) \to Z \times (\Pi_{k=1}^\ell P_k)\times Q^\cZ$ by 
$$\Phi(\varphi):=(\overrightarrow{h}(\varphi), (f_k(\varphi(c_k)))_{1 \leq k \leq \ell}, (g(\varphi(v)))_{v \in \cZ}).$$

{\bf Claim:} $\Phi$ is an $(\varepsilon, \rho_{\infty})$-embedding.\\

Therefore, in light of (\ref{Fk esitmate}) and (\ref{WdimFk}),  we have
\begin{align*}
&\Wdim_\varepsilon(\Map(\rho, F^{-1}, \delta, \sigma), \rho_\infty)\\
&\leq  \dim \left(Z \times (\Pi_{k=1}^\ell P_k)\times Q^\cZ
\right)\\
     &\leq \dim(Z) +\sum_{k=1}^\ell \dim(P_k) +|\cZ|\dim(Q) \\
    & \leq \beta d + \sum_{k=1}^\ell \Wdim_{\varepsilon/3}(X, \rho_{F_k}) +\tau d\Wdim_\varepsilon(X, \rho) \\
    & \leq \beta d + \sum_{k=1}^\ell 2\beta|F_k| +\beta d < 4\beta d
\end{align*}
as desired.

{\it Proof of Claim:}
Suppose that $\varphi, \varphi' \in \Map(\rho, F^{-1}, \delta, \sigma)$ satisfy  $\Phi(\varphi)=\Phi(\varphi')$. We need to verify that  for every $v \in [d]$, one has
\begin{equation} \label{verify}
\rho(\varphi(v), \varphi'(v)) < \varepsilon.
\end{equation}
 For $v \in \cZ$, since $g(\varphi(v))=g(\varphi'(v))$,  by the choice of $g$, the inequality (\ref{verify}) follows.
For $v \notin \cZ$, we split the discussion as follows:\\

{\bf Case 1.}  $\overrightarrow{h}(\varphi)_v \neq \overrightarrow{0} \in [0, 1]^\sU$;\\

By the construction of $\overrightarrow{\zeta}$, we have $\overrightarrow{h}(\varphi)_{v, U}=\zeta_U(\varphi(v)) \neq 0$ for some $U \in \sU$. Since ${\rm supp}(\zeta_U) \subseteq U$, we have  $\varphi(v), \varphi'(v) \in U$. Note that ${\rm diam} (U, \rho) < \varepsilon$. Hence the inequality (\ref{verify}) is verified.\\

{\bf Case 2.}  $\overrightarrow{h}(\varphi)_v = \overrightarrow{0}$.\\

In this case, we write $v=\sigma_s(c)$ for the unique $1 \leq k \leq \ell$ and $s \in F_k$.
Since $f_k$ is an $(\varepsilon/3, \rho_{F_k})$-embedding,  we obtain
$$\rho(s\varphi(c_k), s\varphi'(c_k)) < \varepsilon/3.$$
Since  $\{\zeta_U\}_{U \in \sU}$ is a partition of unity, there exists $U \in \sU$ such that $\zeta_U(\varphi(v)) \neq 0$. By the definition of $\overrightarrow{h}$ it forces that
$$\rho(s^{-1}\varphi(v), \varphi(\sigma_{s^{-1}}(v))) \leq \kappa {\rm \ and \ } \rho(s^{-1}\varphi'(v), \varphi'(\sigma_{s^{-1}}(v))) \leq \kappa.$$
Since  $c_k \in \cB$, we have 
 $$\sigma_{s^{-1}}(v)=\sigma_{s^{-1}}(\sigma_s(c_k))=\sigma_{s^{-1}}((\sigma_{s^{-1}})^{-1}(c_k))=c_k.$$
So
$\rho(s^{-1}\varphi(v), \varphi(c_k))  \leq \kappa$. 
By the choice of $\kappa$, we obtain $\rho(\varphi(v), s\varphi(c_k)) < \varepsilon/3$. Simultaneously, $\rho(\varphi'(v), s\varphi'(c_k)) < \varepsilon/3$ also holds.
Therefore,
\begin{align*}
\rho(\varphi(v), \varphi'(v)) &\leq \rho(\varphi(v), s\varphi(c_k)) +\rho(s\varphi(c_k), s\varphi'(c_k)) +\rho(s\varphi'(c_k), \varphi'(v)) \\
                              & < \varepsilon/3 +\varepsilon/3 + \varepsilon/3 =\varepsilon.
\end{align*}

\end{proof}

\begin{corollary}
    As $\Gamma$ is nonamenable, $\mdims(\Gamma \curvearrowright X) \leq \mdim^\nv(\Gamma \curvearrowright X)$.
\end{corollary}

\begin{proof}
    By Proposition \ref{range}, $\mdim^\nv(\Gamma \curvearrowright X)$ takes a value of either $0$ or $+\infty$. If $\mdim^\nv(\Gamma \curvearrowright X)=0$, by Theorem \ref{main 1}, we have $\mdims(\Gamma \curvearrowright X) \leq 0 =\mdim^\nv(\Gamma \curvearrowright X)$; if $\mdim^\nv(\Gamma \curvearrowright X)=+\infty$, the conclusion is trivial.
\end{proof}

In connection of Theorem \ref{main 1} and Proposition \ref{subgroup discussion} we conclude the following.
\begin{corollary}
    Let $\Gamma$ be a sofic group with a sofic approximation $\Sigma$. Under the assumption in Proposition \ref{subgroup discussion} we have $\mdims(\Gamma \curvearrowright X) \leq 0$.
\end{corollary}

\section{Naive  mean rank}
\numberwithin{equation}{section}
\setcounter{equation}{0}

In this section we introduce the notion of naive mean rank and discuss its properties. In particular, we obtain satisfactory answers for naive mean dimension of algebraic actions. For any (abelian) group $G$, write $\sF(G)$ for the collection of finitely generated subgroups of $G$. 
\begin{definition} \label{nmrk def}
    Let $\cM$ be a $\ZG$-module. For each $\sA \in \sF(\cM)$ and $F \in \cF(\Gamma)$ write $\sA^F=\sum_{s \in F} s^{-1}\sA$.  Set
    $$\mrk^\nv(\sA):=\inf_{F \in \cF(\Gamma)} \frac{\rk(\sA^F)}{|F|}.$$
    We define the {\it naive mean rank} of $\cM$  as
    $$\mrk^\nv(\cM):=\sup_{\sA \in \sF(\cM)} \mrk^\nv(\sA).$$
\end{definition}

We remark that one can also consider a more general approach started with a length function in the spirit of \cite[Remark 3.16]{LL18}.

As $\Gamma$ is amenable, by \cite[Lemma 3.3]{LL18},  function $F \mapsto \rk(\sA^F)$ is strongly subadditive. By Lemma \ref{OW}, we  get a well-defined notation:
$$ \mrk(\sA):= \lim_F \frac{\rk(\sA^F)}{|F|}=\mrk^\nv(\sA).$$
\begin{definition}\cite[Definition 3.5]{LL18}
    For any $\ZG$-module $\cM$ of an amenable group $\Gamma$, the {\it mean rank} of $\cM$ is defined as 
$$\mrk(\cM):=\sup_{\sA \in \sF(\cM)} \mrk(\sA).$$
\end{definition}

By the strength of strong subadditivity, it is immediate that 
\begin{proposition} \label{nmrk basic}
    Let $\cM$ be any $\ZG$-module. If $\Gamma$ is amenable, one has $\mrk^\nv(\cM)=\mrk(\cM)$. 
\end{proposition}

By a similar argument as in the proof of Proposition \ref{range}, we have
\begin{proposition} \label{nmrk basic 2}
For any  $\ZG$-module $\cM$,   if $\Gamma$ is nonamenable, one has $\mrk^\nv(\cM) \in \{0, +\infty\}$. 
\end{proposition}

From the proof of \cite[Theorem 4.1]{LL18}, using a generalized definition of naive mean dimension in Remark \ref{Hausdorff case}, we obtain
\begin{theorem} \label{bridge equality}
    For any group $\Gamma$ and any $\ZG$-module $\cM$, we have $\mdim^\nv(\hcM)=\mrk^\nv(\cM)$.
\end{theorem}

\begin{corollary}\label{amenable match}  
Let $\Gamma \curvearrowright X$ be an algebraic action of an amenable group. Then 
$\mdim^\nv(X)=\mdim(X)$
\end{corollary} 

\begin{proof}
  By \cite[Theorem 4.1]{LL18}, Proposition \ref{nmrk basic}, and Theorem \ref{bridge equality}, we infer that
    $$\mdim(X)=\mrk(\widehat{X})=\mrk^\nv(\widehat{X})=\mdim^\nv(X).$$
\end{proof}

Now we compare naive mean rank and sofic mean rank. Let us first recall the definition of sofic mean rank. 
Let $\cM$ be any $\ZG$-module and $\sA, \sB \in \sF(\cM)$. For each $d \in \bN$ we may consider $\cM^d \cong \bZ^d \otimes_\bZ \cM$ as abelian groups and write the element of $\cM^d$ as
$$\sum_{v \in [d]}\delta_v\otimes a_v$$
for some $v \in [d]$ and $ a_v \in \cM$, where $\delta_v$ denotes the element of $\bZ^d$ taking value 1 at the coordinate $v$ and $0$ everywhere else.  For any $F \in \cF(\Gamma)$ and any map $\sigma: \Gamma \to \Sym(d)$, denote by $\sM(\sB, F, \sigma)$ the abelian subgroup of $\cM^d$ generated by the elements 
$$\delta_v\otimes b-\delta_{\sigma(s)v}\otimes sb$$
for all $v \in [d], b \in \sB$, and $s \in F$.  Put $\sM(\sA, \sB, F, \sigma)$ as the image of $\sA^d$ in $\cM^d/\sM(\sB, F, \sigma)$ under the natural quotient map $\cM^d \to \cM^d/\sM(\sB, F, \sigma)$. 

Fix $\Sigma=\{\sigma_i \colon \Gamma \to \Sym(d_i)\}_{i=1}^\infty$ as a sofic approximation of a sofic group $\Gamma$.
\begin{definition}\cite[Definition 3.1]{LL19} \label{sofic algebraic entropy}
Let $\cM_1 \subseteq \cM_2$ be two $\ZG$-modules. For any $\sA \in \sF(\cM_1), \sB \in \sF(\cM_2)$ and $F \in \cF(\Gamma)$, set
$$\mrk(\sA|\sB, F):=\varlimsup_{i \to +\infty} \frac{\rk(\sM(\sA, \sB, F, \sigma_i))}{d_i},$$
and 
$$\mrk(\sA|\cM_2):= \inf_{F \in \cF(\Gamma)} \inf_{\sB \in \sF(\cM_2)}\mrk(\sA|\sB, F). $$
The {\it sofic mean rank of $\cM_1$ relative to $\cM_2$}  is defined as
$$\mrks(\cM_1|\cM_2):=\sup_{\sA \in \sF(\cM_1)} \mrk(\sA|\cM_2).$$
The {\it sofic mean rank} of $\cM_1$ is then defined as 
$$\mrks(\cM_1):=\mrks(\cM_1|\cM_1).$$
\end{definition}

In light of \cite[Proposition 8.5]{LL19} and \cite[Theorem 10.1]{LL19}, we have
\begin{theorem} \label{sofic coincide}
    For any algebraic action $\Gamma \curvearrowright X$ of a sofic group, one has $\mdims(X)=\mrks(\widehat{X})$.
\end{theorem}

Now we give a self-contained proof of Theorem \ref{nmrk compare}. With the help of Theorems \ref{main 1} and \ref{sofic metric coincide}, one can also give a quick proof.
\begin{proof}[Proof of Theorem \ref{nmrk compare}]
     If $\Gamma$ is amenable, by \cite[Theorem 5.1]{LL19} and Proposition \ref{nmrk basic}, we have
  $$\mrks(\cM)=\mrk(\cM)=\mrk^\nv(\cM).$$
  To deal with the nonamenable case, based on Proposition \ref{nmrk basic 2}, it suffices to show that $\mrk^\nv(\cM)=0$ forces that $\mrks(\cM)=0$. Fix $\sA \in \cF(\cM)$  and any $\beta > 0$. It is reduced to show
$$\mrks(\sA|\cM) \leq 3\beta.$$

Take $\tau > 0$ such that $\rk(\sA) \leq \beta/\tau$. Since $\mrk^\nv(\sA)=0$, there exists  $F \in \cF(\Gamma)$ such that 
\begin{equation} \label{sA^F rank esitmate}
 \rk(\sA^F) \leq \tau\beta|F|.
\end{equation}
Applying Lemma \ref{assigning} to $\eta=0$, as $\sigma \colon \Gamma \to \Sym(d)$ is good enough admitting a set $\cB \subseteq [d]$ satisfying $|\cB| \geq (1-\frac{\tau}{2(|F|+1)})$ and 
    $$\sigma_s(v)\neq \sigma_t(v), \ \sigma_{s^{-1}}(v)=(\sigma_s)^{-1}(v)$$
for all $v \in \cB$ and $s\neq t \in F\cup F^{-1}$, there exist $\ell \in \bN$, $F_1, \dots, F_\ell \subseteq F$, $c_1, \dots, c_\ell \in \cB$ such that
    \begin{itemize}
\item[i)] for every $k =1, \dots, \ell$, one has $|F_k|\geq \tau|F|/2 $;\\
\item[ii)] the sets $\{\sigma(F_k)c_k\}_{k=1}^\ell$ are pairwise disjoint with $|\bigsqcup_{k=1}^\ell \sigma(F_k)c_k| \geq (1-\tau)d$.
\end{itemize}

Once again for any $F' \subseteq F$ with $|F'| \geq \tau|F|/2$, by (\ref{sA^F rank esitmate}), one has
$$\rk(\sA^{F'}) \leq \rk(\sA^F) \leq \tau\beta|F| \leq 2\beta|F'|.$$
In particular, for every $1\leq k \leq \ell$, we have
\begin{equation} \label{Fk rank esitmate}
    \rk(\sA^{F_k}) \leq 2\beta|F_k|.
\end{equation}

Put $\sB=\sA^F$, $\sA_0=\sum_{v \in [d]\setminus \sqcup_{k=1}^\ell \sigma(F_k)c_k} \delta_v\otimes \sA$, and $\sA_k=\delta_{c_k}\otimes \sA^{F_k}$ for every $k=1, \dots, \ell$. Consider the natural quotient map $\Phi \colon \cM^d \to \cM^d/\sM(\sB, F, \sigma)$. Note that for any $s \in F_k, a \in \sA$ one has
$$\Phi(\delta_{\sigma(s)c_k}\otimes a)=\Phi(\delta_{c_k}\otimes s^{-1}a) \in \Phi(\sA_k).$$
It concludes that 
$\sM(\sA, \sB, F, \sigma) \subseteq \Phi(\sA_0) +\sum_{k=1}^\ell \Phi(\sA_k)$.
Using assumption on $\rk(\sA)$ and (\ref{Fk rank esitmate}) we obtain
\begin{align*}
    \rk(\sM(\sA, \sB, F, \sigma)) & \leq \rk(\sA_0) + \sum_{k=1}^\ell \rk(\sA_k) \\
    & \leq \rk(\sA)\tau d +2\beta d \\
    & \leq \beta d +2\beta d =3\beta d.
\end{align*}
Thus $\mrks(\sA|\cM) \leq \mrks(\sA|\sB, F) \leq 3\beta$ as desired.  
\end{proof}


\begin{proof}[Proof of Corollary \ref{main 2}]
    Combining Theorems \ref{sofic coincide}, \ref{nmrk compare}, with \ref{bridge equality}, we infer that 
    $$\mdims(X)=\mrks(\widehat{X})\leq \mrk^\nv(\widehat{X})=\mdim^\nv(X).$$
\end{proof}


\section{Naive  metric mean dimension}
\numberwithin{equation}{section}
\setcounter{equation}{0}
In this section, we introduce the notion of naive metric mean dimension and mainly  discuss its relation with naive metric mean dimension.

Consider a continuous pseudometric $\rho$ on a compact metrizable space $X$ and any $\varepsilon >0$. Recall that a subset $Z$ of $X$ is called {\it $(\varepsilon, \rho)$-separated} if $\rho(z_1, z_2) \geq \varepsilon$ for all distinct $z_1, z_2 \in Z$; $Z$ is called {\it $(\varepsilon, \rho)$-spanning} if for each $x \in X$ there exists $z \in Z$ such that $\rho(x, z) < \varepsilon$. Set $N_\varepsilon(X, \rho):=\max_Z|Z|$ for $Z$ ranging over all $(\varepsilon, \rho)$-separated subsets of $X$.

\begin{definition}
Let $\Gamma \curvearrowright X$ be a continuous action of any countable group $\Gamma$. Set
$$h_\varepsilon^\nv(\rho)=\inf_{F \in \cF(\Gamma)} \frac{\log N_\varepsilon(X, \rho_F)}{|F|}.$$
When $\rho$ is compatible, we define the {\it naive metric mean dimension of $\Gamma \curvearrowright X$ with respect to $\rho$} as 
$$\mdim_\rM^\nv(\Gamma \curvearrowright X, \rho):=\varliminf_{\varepsilon \to 0} \frac{1}{\log(1/\varepsilon)} h_\varepsilon^\nv(\rho).$$
We also define the {\it naive metric mean dimension of $\Gamma \curvearrowright X$} as
$$\mdim_\rM^\nv(\Gamma \curvearrowright X):=\inf_\rho \mdim_\rM^\nv(\Gamma \curvearrowright X, \rho)$$
for $\rho$ ranging over compatible metrics on $X$.
Meanwhile, the {\it naive topological entropy of $\Gamma \curvearrowright X$ } is defined as 
$$h^\nv(\Gamma \curvearrowright X ):=\sup_{\varepsilon > 0} h_\varepsilon^\nv(\rho),$$
which is independent of the choice of compatible metric $\rho$ (see \cite[Definition 2.4]{B17}).
\end{definition}

 If the space $X$ and the acting group $\Gamma$ is understood in the context, we may write $\mdim_\rM^\nv(X, \rho)$ and $h^\nv(X)$ for $\mdim_\rM^\nv(\Gamma \curvearrowright X, \rho)$ and $h^\nv(\Gamma \curvearrowright X)$ respectively in short.

Now we recall the definition of sofic metric mean dimension. Fix a sofic approximation $\Sigma=\{\sigma_i \colon \Gamma \to \Sym(d_i)\}_{i=1}^\infty$  of a sofic group $\Gamma$.
\begin{definition}\cite[Definition 4.1]{Li13}
 For a  continuous action $\Gamma \curvearrowright X$ of a sofic group, define
$$h_\Sigma^\varepsilon(\rho):=\inf_{F \in \cF(\Gamma)} \inf_{\delta > 0} \varlimsup_{i\to \infty}\frac{1}{d_i}\log N_\varepsilon(\Map(\rho, F, \delta, \sigma_i),  \rho_{\infty}). $$
When $\varepsilon, F, \delta$ are fixed, denote the above limit supremum as $h_\Sigma^\varepsilon(\rho, F, \delta)$. If  $\Map(\rho, F, \delta, \sigma_i)$ is empty for all sufficiently large $i$, the limit supremum is set to be $-\infty$.

When $\rho$ is compatible, the {\it sofic metric mean dimension of $\Gamma \curvearrowright X$ with respect to $\rho$} is defined as
$$\mdimsm(\Gamma \curvearrowright X, \rho):= \varliminf_{\varepsilon \to 0} \frac{h_\Sigma^\varepsilon(\rho)}{|\log \varepsilon|}.$$
The {\it sofic metric mean dimension of $\Gamma \curvearrowright X$} is defined as
$$\mdimsm(\Gamma \curvearrowright X, \rho):= \inf_{\rho} \mdimsm(\Gamma \curvearrowright X, \rho),$$
for $\rho$ ranging over compatible metrics on $X$. 
Meanwhile, the {\it sofic topological entropy of $\Gamma \curvearrowright X$ } is defined as 
$$h_\Sigma(\Gamma \curvearrowright X ):=\sup_{\varepsilon > 0} h_\Sigma^\varepsilon(\rho)$$
for some compatible metric $\rho$ (see \cite[Definition 2.2]{KL13}).

\end{definition}

Modifying argument of \cite[Theorem 1.1]{B17} and \cite[Lemma 5.3]{Li13}, we prove
\begin{theorem} \label{decaying case}
Let $\Gamma \curvearrowright X$ be any continuous action of a sofic group $\Gamma$ and $\rho$ a compatible metric on $X$. Suppose that  $\mdim_\rM^\nv(\Gamma \curvearrowright X, \rho) =0$ and furthermore 
$$\varliminf_{\varepsilon \to 0} \frac{h^\nv_{\varepsilon}(\rho)}{|\log \varepsilon|} \frac{\log N_{\varepsilon}(X, \rho)}{|\log\varepsilon|}=0. $$  (decaying speed of $\varepsilon$-entropy over $|\log\varepsilon|$ is far greater than the expanding speed of $\varepsilon$-box dimension). Then $\mdimsm(\Gamma \curvearrowright X, \rho)\leq 0$.
\end{theorem}

\begin{proof}
Fix $\beta > 0$. We shall prove that $\mdimsm(X, \rho) \leq 3\beta$. From the assumption there exist positive numbers $\varepsilon_i \to 0$ such that 
$$\frac{h^\nv_{\varepsilon_i}(\rho)}{|\log \varepsilon_i|}=o\left(\frac{|\log \varepsilon_i|}{\log N_{\varepsilon_i}(X, \rho)}\right).$$ Thus as $i$ is large enough, we have 
$$h^\nv_{\varepsilon_i}(\rho) < |\log \varepsilon_i|\tau_i\beta,$$
where $\tau_i$ sits inside $(0, 1)$ such that $N_{\varepsilon_i}(X, \rho)^{2\tau_i}=(1/\varepsilon_i)^\beta$. Furthermore, there exists $F^{(i)} \in \cF(\Gamma)$ such that 
\begin{equation} \label{separated control}
  \log N_{\varepsilon_i}(X, \rho_{F^{(i)}}) < |\log \varepsilon_i|\tau_i\beta |F^{(i)}|.
\end{equation}
Fixing $i$ large enough such that the above inequality holds, for simplicity, we may write $\varepsilon=\varepsilon_i, \tau=\tau_i$, and $F=F^{(i)}$ in the sequel.

By Lemma \ref{assigning}, as a map $\sigma \colon \Gamma \to \Sym(d)$ is good enough, for every $\cW \subseteq [d]$ with $|\cW| \geq (1 - \frac{\tau}{|F|+1})d$, there exist $\ell \in \bN, F_1, \dots, F_\ell \subseteq F, c_1, \dots, c_\ell \in \cW$, such that each $|F_k|\geq \tau|F|/2 $ for each $k =1, \dots, \ell$, and 
the sets $\{\sigma(F_k)c_k\}_{k=1}^\ell$ are pairwise disjoint with $|\bigsqcup_{k=1}^\ell \sigma(F_k)c_k| \geq (1-2\tau)d$.

Observe that for every $F' \subseteq F$ with $|F'|\geq \tau|F|/2$, by inequality (\ref{separated control}), we conclude
$$\log N_\varepsilon(X, \rho_{F'}) \leq \log N_\varepsilon(X, \rho_F) \leq |\log \varepsilon|\tau \beta |F| \leq 2|\log\varepsilon|\beta|F'|.$$
In particular, for every $F_k, k=1, \dots, \ell$, we have
\begin{equation} \label{F_k control}
    \log N_\varepsilon(X, \rho_{F_k}) \leq 2|\log \varepsilon||F_k|\beta.
\end{equation}

Let $\delta=\delta_i \in (0, 1/2|F^{(i)}|)$ be small enough (to be further determined). We shall show as $\sigma$ is good enough,
$$N_{4\varepsilon_i}(\Map(\rho, F^{(i)}, \delta_i, \sigma), \rho_\infty) \leq e^{\beta d}\left(\frac{1}{\varepsilon_i}\right)^{3\beta d}.$$
It follows that
$$h_\Sigma^{4\varepsilon_i}(\rho) \leq h_\Sigma^{4\varepsilon_i}(\rho, F^{(i)}, \delta_i) \leq 3\beta|\log \varepsilon_i| +\beta$$
and hence
$$\mdimsm(\Gamma \curvearrowright X, \rho) \leq \varliminf_{i \to \infty} \frac{h_\Sigma^{4\varepsilon_i}(\rho)}{|\log(4\varepsilon_i)|} \leq 3\beta$$
as desired.

Choose $\cE \subseteq \Map(\rho, F, \delta, \sigma)$ as a maximal $(4\varepsilon, \rho_\infty)$-separated subset. We shall pick a nice subset $\cF \subseteq \cE$ such that the set
$$\cW_\varphi:=\{v \in [d]: \rho(s\varphi(v), \varphi(\sigma_s(v))) \leq \sqrt{\delta} {\rm \ for \ every \ } s \in F\}$$
stay the same for every $\varphi \in \cF$.
 Observe that for every $\varphi \in \Map(\rho, F, \delta, \sigma)$, by Lemma \ref{Map lowerbound}, we have $|\cW_\varphi| \geq (1 -|F|\delta)d$. By Stirling's approximation \cite[Lemma 10.1]{KLb}, as $d \to \infty$, for any $\lambda \in (0, 1)$,  there exists $\eta=\eta(\lambda) > 0$ such that 
$\lambda d\binom{d} {\lambda d} < e^{\eta d}$
and $\eta(\lambda)$ tends to zero as $\lambda \to 0$.
In particular, for $\lambda=|F|\delta$, as $d$ is sufficiently enough, we have
$$|\{\cV \subseteq [d]: |\cV| \geq (1-\lambda)d\}| \leq \sum_{j=0}^{\lambda d} \binom{d}{j} \leq \lambda d\binom{d} {\lambda d} < e^{\eta d}$$
for some $\eta=\eta(F, \delta) > 0$ that is independent of $d$. Note that for fixed $F$, as $\delta$ is small enough, $\eta$ can be arbitrarily small and hence $\eta < \beta$. 
It concludes that when $\delta $ is small enough and $d$ is sufficiently large, there exists a subset $\cF \subseteq \cE$ satisfying that
   \begin{itemize}
\item[i)] $|\cF| \geq |\cE|e^{-\beta d}$;\\
\item[ii)] the set $\cW_\varphi$ stays the same for every $\varphi \in \cF$, which we denote by $\cW$;\\
\item[iii)] $|\cW| \geq (1-\frac{\tau}{|F|+1})d$,  as $|F|\sqrt{\delta} < \frac{\tau}{|F|+1}$.
\end{itemize}
Applying Lemma \ref{assigning}, we obtain $F_k, c_k \in \cW$ as stated. 

For each $\mathcal{A} \subseteq [d]$ denote by $\rho_{\mathcal{A}}$ the pseudometric on $X^{[d]}$ by
$$\rho_\mathcal{A}(\varphi, \varphi')=\max_{v \in \mathcal{A}} \rho(\varphi(v), \varphi'(v)).$$
For each $k=1, \dots, \ell$, choose $\cD_k \subseteq \cF$ as a maximal $(2\varepsilon, \rho_{\sigma(F_k)c_k})$-separated subset of $\cF$.\\

{\bf Claim 1.} As $2\sqrt{\delta} < \varepsilon$, the set $\{\varphi(c_k): \varphi \in \cD_k\}$ is $(\varepsilon, \rho_{F_k})$-separated.\\

Consequently, by the inequality (\ref{F_k control}), we obtain
\begin{equation} \label{D_k control}
    |\cD_k|\leq N_\varepsilon(X, \rho_{F_k}) \leq e^{2|F_k|\beta|\log \varepsilon|}=\left(\frac{1}{\varepsilon}\right)^{2\beta |F_k|}.
\end{equation}

Now let $A \subseteq X$ be a maximal $(2\varepsilon, \rho)$-separated subset. Put $\cZ=[d]\setminus \sqcup_{k=1}^\ell \sigma(F_k)c_k$. Then $A^\cZ$ is a $(2\varepsilon, \rho_\infty)$-spanning subset of $X^\cZ$. Since $\cD_k \subseteq \cF$ is $(2\varepsilon, \rho_{\sigma(F_k)c_k})$-spanning,  for each $\varphi \in \cF$, there exists $\varphi_k \in \cD_k$ such that $\rho_{\sigma(F_k)c_k}(\varphi, \varphi_k) < 2\varepsilon$.
Meanwhile, since $A^\cZ$ is $(2\varepsilon, \rho_\infty)$-spanning in $X^\cZ$, there exists $\varphi_0 \in A^\cZ$ such that $\rho_\infty(\varphi|_\cZ, \varphi_0) < 2\varepsilon$. It concludes that 
$$\rho_\infty(\varphi, ((\varphi_k|_{\sigma(F_k)c_k})_{1\leq k \leq \ell}, \varphi_0)) < 2\varepsilon.$$
Thus we obtain a map $\Phi \colon \cF \to A^\cZ \times \Pi_{k=1}^\ell \cD_k$ sending $\varphi$ to $(\varphi_0, (\varphi_k)_{1 \leq k \leq \ell})$.\\

{\bf Claim 2.} The map $\Phi$ is injective.\\

Therefore, by the inequality (\ref{D_k control}), we obtain
\begin{align*}
    N_{4\varepsilon}(\Map(\rho, F, \delta, \sigma), \rho_\infty)=&|\cE| \leq e^{\beta d} |\cF|\\
    &\leq  e^{\beta d} |A|^{|\cZ|}\Pi_{k=1}^\ell |\cD_k|\\
    &\leq  e^{\beta d} N_\varepsilon(X, \rho)^{2\tau d}\Pi_{k=1}^\ell N_\varepsilon(X, \rho_{F_k})\\
    &\leq e^{\beta d} \left(\frac{1}{\varepsilon}\right)^{\beta d+2\beta\sum_{k=1}^\ell |F_k|}  \\
    &\leq e^{\beta d}  \left(\frac{1}{\varepsilon_i}\right)^{3\beta d}
\end{align*}
as desired.

{\it Proof of Claim 1.} For every $\varphi, \varphi' \in \cD_k$ and $s \in F_k$, we want to show 
$$\rho(s\varphi(c_k), s\varphi'(c_k)) \geq \varepsilon.$$
Since $c_k \in \cW_\varphi=\cW_{\varphi'}$ and $\cD_k$ is $(2\varepsilon, \rho_{\sigma(F_k)c_k})$-separated, we have
\begin{align*}
    &\rho(s\varphi(c_k), s\varphi'(c_k))\\
    &\geq \rho(\varphi(\sigma_s(c_k)), \varphi'(\sigma_s(c_k))) -\rho(\varphi(\sigma_s(c_k)), s\varphi(c_k)) -\rho(\varphi'(\sigma_s(c_k)), s\varphi'(c_k))\\
    &\geq 2\varepsilon -2\sqrt{\delta} > \varepsilon.
\end{align*}

{\it Proof of Claim 2.} Assume $\Phi(\varphi)=\Phi(\varphi')$. By the definition of $\Phi$,  we have $\rho_\infty(\varphi|_\cZ, \varphi_0) < 2\varepsilon$ and $\rho_\infty(\varphi_0, \varphi'|_\cZ)=\rho_\infty(\varphi'_0, \varphi'|_\cZ) < 2\varepsilon$.
It follows that $\rho_\cZ(\varphi, \varphi') < 4\varepsilon$.
Similarly, we infer that $\rho_{\sigma(F_k)c_k}(\varphi, \varphi') < 4\varepsilon$. In summary, we obtain $\rho_\infty(\varphi, \varphi') < 4\varepsilon$. Since $\cF \subseteq \cE$ is $(4\varepsilon, \rho_\infty)$-separated, it forces that $\varphi=\varphi'$.

\end{proof}

We remark that Theorem \ref{decaying case} is also proved by Mantang Wei (personal communication).


By a slight modification of \cite[Theorem 4.2]{LW00}, we have
\begin{theorem} \label{metric comparison}
    For any continuous action $\Gamma \curvearrowright X$ and a compatible metric $\rho$ on $X$, we have
    $\mdim^\nv(X) \leq \mdim_\rM^\nv(X, \rho)$. In particular, if $h^\nv(X) < \infty$, one has $\mdim^\nv(X)=0$.
\end{theorem}

\begin{corollary}
If $\Gamma$ is amenable group and $h^\nv(X) < \infty$, then 
$$\mdim^\nv(X)=0=\mdim(X)=\mdim_\rM^\nv(X, \rho)$$
for any compatible metric $\rho$ on $X$;
If $\Gamma$ is nonamenable and $\mdim^\nv(X) > 0$, then 
$$\mdim^\nv(X)=+\infty =\mdim_\rM^\nv(X, \rho)$$
 for any compatible metric $\rho$ on $X$.
\end{corollary}

\begin{remark}
     Similar to the naive mean dimension,  naive metric mean dimension can not be computed via dynamically generating pseudometric as well. For example, consider the full shift $\Gamma \curvearrowright [0, 1]^\Gamma$ in Remark \ref{pseudo example} again.  Then $ \mdim_\rM^\nv([0,1]^\Gamma, \rho) \leq 1$. However, by Theorem \ref{metric comparison} and Proposition \ref{nonamenable example}, whenever $\Gamma$ is nonamenable, we have 
   $$\mdim_\rM^\nv([0,1]^\Gamma, \widetilde{\rho}) \geq \mdim^\nv([0, 1]^\Gamma)=\infty.$$ 
\end{remark}

The following says that naive metric mean dimension does not decrease under taking factors in the situation that the factor map is Lipschitz.
\begin{proposition} \label{nondecrasing}
    Let $\pi \colon X \to Y$ be any factor map. Suppose that $\rho_X, \rho_Y$ are two compatible metrics on $X$ and $Y$ respectively such that $\pi$ is Lipschitz with respect to $\rho_X$ and $\rho_Y$, i.e. there exists $c > 0$ such that $\rho_Y(\pi(x), \pi(x')) \leq c\rho_X(x, x')$ for all $x, x' \in X$. Then we have
    $$\mdim_\rM^\nv(X, \rho_X) \geq \mdim_\rM^\nv(Y, \rho_Y).$$
\end{proposition}

\begin{proof}
    Let $\varepsilon > 0$ and $F \in \cF(\Gamma)$. Lift an $(\varepsilon, \rho_{Y, F})$-separated subset $E$ of $Y$ up to a subset $\widetilde{E} \subseteq \pi^{-1}(E)$ such that $|\widetilde{E}|=|E|$. By our assumption on metrics it follows that $\widetilde{E}$ is $(\varepsilon/c, \rho_{X, F})$-separated and hence $N_{\varepsilon/c}(X, \rho_{X, F}) \geq N_\varepsilon(Y, \rho_{Y,F})$. Taking infimum over $F \in \cF(\Gamma)$ we obtain $$h^\nv_{\varepsilon/c}(\rho_X) \geq h^\nv_\varepsilon(\rho_Y),$$
    which finishes the proof.
\end{proof}

\begin{example}
    We exhibit an example with positive naive metric mean dimension but zero sofic metric mean dimension. For free group $\Gamma=F_2$ with two generators $a$ and $b$, consider $\cM_1=\ZG$ as a $\ZG$-submodule of $\cM_2:=\ZG/\ZG(a-1)$. Correspondingly it induces a canonical factor map $\pi \colon X:=\widehat{\cM_2} \to Y:=\widehat{\cM_1}$ between Pontryagin dual groups. From the proof of \cite[Lemma 10.5]{LL19}, there exist dynamically generating  pseudometric $\rho_X$ and  $\rho_Y$ on $X$ and $Y$ respectively such that $\rho_Y(\pi(x), \pi(x')) \leq \rho_X(x, x')$ for all $x, x' \in X$. It follows that $\pi$ is Lipschitz with respect to $\widetilde{\rho_X}$ and $\widetilde{\rho_Y}$. By Proposition \ref{nondecrasing}, we obtain
    $$\mdim_\rM^\nv(X, \widetilde{\rho_X}) \geq \mdim_\rM^\nv(Y, \widetilde{\rho_Y})\geq \mdim^\nv(Y)=\infty.$$

    On the other hand, from \cite[Lemma 10.7]{LL19} and \cite[Example 6.3]{LL19}, we have 
    $$\mdim_{\Sigma, \rM}(X, \widetilde{\rho_X})=\mrk_\Sigma(\cM_2)=0.$$
\end{example}

In light of \cite[Theorem 1.3]{LL19} and \cite[Remark 10.2]{LL19}, we have
\begin{theorem} \label{sofic metric coincide}
    For any algebraic action $\Gamma \curvearrowright X$ of a sofic group, one has $\mdims(X)=\mdimsm(X, \rho)$ for some compatible metric $\rho$ on $X$.
\end{theorem}

\begin{proposition} \label{algebraic action comparision}
    Let $\Gamma \curvearrowright X$ be an algebraic action of a sofic group $\Gamma$. Then 
    $$\mdimsm(\Gamma \curvearrowright X) \leq \mdim_\rM^\nv(\Gamma \curvearrowright X).$$
\end{proposition}

\begin{proof}
By Theorem \ref{sofic metric coincide}, Corollary \ref{main 2} and Theorem \ref{metric comparison}, we conclude
$$\mdimsm(\Gamma \curvearrowright 
X)=\mdims(\Gamma \curvearrowright X)\leq \mdim^\nv( \Gamma \curvearrowright X)\leq \mdim_\rM^\nv(\Gamma \curvearrowright X, \rho)$$
for any compatible metric  $\rho$ on $X$.



\end{proof}

The following result provides a partial answer to a question raised by Xiaojun Huang.
\begin{proposition}
    Let $\Gamma \curvearrowright X$ be an algebraic action by an amenable group. Then $$\mdim^\nv(X)=\mdim_\rM^\nv(X, d)$$
for some compatible metric $d$ on $X$.
\end{proposition}

\begin{proof}
    Observe that by definition, for any compatible metric $\rho$ on $X$, one has
    $$\mdim_\rM^\nv(X, \rho) \leq \mdim_\rM(X, \rho).$$
 By \cite[Theorem 7.2]{LL18}, there exists a compatible metric $d$ on $X$ such that  $$\mdim_\rM(X, d)=\mdim(X).$$
 Thus in light of Theorems  \ref{metric comparison}
and Corollary \ref{amenable match}, the conclusion follows.
\end{proof}

\section{Naive small boundary property}
To end this paper, we consider a natural notion of naive small boundary property in connection with systems of zero mean dimension.   Let $\Gamma \curvearrowright X$ be any continuous action on a compact metrizable space by a countable group $\Gamma$.
\begin{definition}
   For any set $A \subseteq X$, we define the {\it naive orbit capacity of $A$} as
    $$\ocap^\nv(A):=\inf_{F \in \cF(\Gamma)} \sup_{x \in X} \frac{\sum_{s \in F} 1_A(sx)}{|F|}.$$
    We say that the action $\Gamma \curvearrowright X$ has the {\it naive small-boundary property}  if for every point $x \in X$ and every neighborhood $U$ of $x$ there exists a neighborhood $V \subseteq U$ of $x$ such that $$\ocap^{\nv}(\partial V)=0.$$
\end{definition}

As $\Gamma$ is amenable, Lindenstrauss-Weiss defined that $\Gamma \curvearrowright X$ has the {\it small boundary property} if for every point $x \in X$ and every neighborhood $U$ of $x$ there exists a neighborhood $V \subseteq U$ of $x$ such that $$\ocap(\partial V)=0.$$
Here in view of Lemma \ref{OW}, the {\it orbit capacity of $A$} is defined as
    $$\ocap(A):=\lim_F \sup_{x \in X} \frac{\sum_{s \in F} 1_A(sx)}{|F|}.$$

Denote by $M_\Gamma(X)$ the set of all $\Gamma$-invariant Borel probability measures on $X$. Recall that in \cite[Definition 8.1]{Li13}, Li defined that a continuous action $\Gamma \curvearrowright X$ has the {\it small-boundary property} (SBP) if for every point $x \in X$ and every neighborhood $U$ of $x$ there exists a neighborhood $V \subseteq U$ of $x$ such that  
$$\mu(\partial V)=0$$
for every $\mu \in M_\Gamma(X)$.

By definition it is clear that the naive SBP always implies Li's SBP (For any $\mu(A) \in M_\Gamma(X)$, by the invariance of $\mu$, one gets $\mu(A) \leq \ocap^\nv(A)$ for any  Borel set $A \subseteq X$). From \cite[Lemma 2.9]{DHZ19} and  \cite[Lemma 3.2]{KS20}, we see that as $\Gamma$ is amenable, naive SBP coincides with the SBP, and Li's SBP. 
Since SBP implies the system is of zero mean dimension, we conclude that 
$$\mdim(\Gamma \curvearrowright X)=0=\mdim^\nv(\Gamma \curvearrowright X).$$
for any continuous action $\Gamma \curvearrowright X$ with SBP.
 Conversely one may ask whether systems of zero mean dimension imply the SBP. Such a situation holds for $\Gamma=\mathbb{Z}^d$ and $\Gamma \curvearrowright X$ has the Marker property \cite[Corollary 5.4]{GLT16} or more generally $\Gamma \curvearrowright X$ is free and has the uniform Rokhlin property \cite[Theorem 5.1]{N21}. 

\begin{question}
    Does Li's SBP imply naive SBP? Is there any counterexample for this implication?
\end{question}

Modifying the proof of \cite[Theorem 5.4]{LW00} we obtain
\begin{theorem}
    If a continuous action $\Gamma \curvearrowright X$ has the naive small boundary property, then $\mdim^\nv(\Gamma \curvearrowright X)=0$. 
\end{theorem}


\begin{thebibliography}{99}
\Small





\bibitem[B16]{B16}
L.~Bowen. Zero entropy is generic. {\it Entropy} {\bf 18} (2016), no. 6, Paper No. 220, 20 pp.

\bibitem[B20]{B20}
L.~Bowen. Examples in the entropy theory of countable group actions. {\it Ergodic Theory Dynam. Systems} {\bf 40} (2020), no. 10, 2593--2680. 

\bibitem[B17]{B17}
P.~Burton. Naive entropy of dynamical systems. {\it Israel J. Math.} {\bf 219} (2017), no. 2, 637--659.

\bibitem[CC10]{CC10}
T.~Ceccherini-Silberstein and M.~Coornaert. {\it Cellular Automata and Groups}. Springer Monographs in Mathematics. Springer-Verlag, Berlin, 2010.

\bibitem[CDZ22]{CDZ22}
E.~Chen,  D.~Dou, and D.~Zheng. Variational principles for amenable metric mean dimensions. {\it J. Differential Equations} {\bf 319} (2022), 41--79.



 \bibitem[C15]{C15}
M.~Coornaert. {\it Topological Dimension and Dynamical Systems}. Translated and revised from the 2005 French original. Universitext. Springer, Cham, 2015.


\bibitem[CLb]{CLb}
V. ~Capraro and M. ~Lupini. {\it Introduction to sofic and hyperlinear groups and Connes' embedding conjecture.} Lecture Notes in Math., 2136. Springer, Berlin, 2015.

\bibitem[CLS21]{CLS21}
D.~Cheng, Z.~Li, and B.~Selmi. Upper metric mean dimensions with potential on subsets. {\it Nonlinearity} {\bf 34} (2021), no. 2, 852--867.

\bibitem[DFR16]{DFR16}
T.~Downarowicz, Tomasz, B.~Frej, and P.-P.~Romagnoli.  Shearer's inequality and infimum rule for Shannon entropy and topological entropy.{\it Dynamics and numbers}, 63--75, {\it Contemp. Math.}, {\bf 669}, {\it Amer. Math. Soc., Providence, RI}, 2016. 


\bibitem[DHZ19]{DHZ19}
T.~Downarowicz, D.~Huczek and G.~Zhang. Tilings of amenable groups. {\it J. Reine Angew. Math.} {\bf 747} (2019), 277--298.


\bibitem[EZ17]{EZ17}
 G.~A.~Elliott and Z.~Niu. The $C^\ast$-algebra of a minimal homeomorphism of zero mean dimension. {\it Duke Math. J.} {\bf 166} (2017), no. 18, 3569--3594.


\bibitem[G99m]{G99m}
M. ~Gromov. Topological invariants of dynamical systems and spaces of holomorphic maps. I. {\it  Math. Phys. Anal. Geom.} {\bf 2} (1999), no. 4, 323--415.

 \bibitem[G99s]{G99s}
M.~Gromov. Endomorphisms of symbolic algebraic varieties. {\it J. Eur.\ Math.\ Soc.} {\bf 1} (1999), no. 2, 109--197.

\bibitem[G15]{G15}
Y.~Gutman. Mean dimension and Jaworski-type theorems. {\it Proc. Lond. Math. Soc. (3)} {\bf 111} (2015), no. 4, 831--850.

\bibitem[GLT16]{GLT16}
Y.~Gutman, E.~Lindenstrauss, and M.~Tsukamoto. Mean dimension of $\bZ^k$-actions.
{\it Geom. Funct. Anal.} {\bf 26} (2016), no. 3, 778--817.

\bibitem[GQT19]{GQT19}
Y.~Gutman, Y.~Qiao, and M.~Tsukamoto. Application of signal analysis to the embedding problem of $\bZ^k$-actions. {\it Geom. Funct. Anal.} {\bf 29} (2019), no. 5, 1440--1502.

\bibitem[GT20]{GT20}
Y.~Gutman and M.~Tsukamoto. Embedding minimal dynamical systems into Hilbert cubes. {\it Invent. Math.} {\bf 221} (2020), no. 1, 113--166.


\bibitem[JQ21]{JQ21}
L.~Jin and Y.~Qiao. Sofic mean dimension of typical actions and a comparison theorem. {\it arXiv:2108.02536}.


 \bibitem[KL13]{KL13}
D. ~Kerr and H. ~Li. Combinatorial independence and sofic entropy. {\it Commun. Math. Stat.} {\bf 1} (2013), no. 2, 213--257.

\bibitem[KLb]{KLb}
D.~Kerr and H.~Li. {\it Ergodic Theory: Independence and Dichotomies.} Springer Monographs in Mathematics. Springer, Cham, 2016.

\bibitem[KS20]{KS20}
D.~Kerr and  G.~Szab{\'o}. Almost finiteness and the small boundary property. {\it Comm. Math. Phys.} {\bf 374} (2020), no. 1, 1--31.


\bibitem[Li13]{Li13}
H.~Li. Sofic mean dimension. {\it Adv. Math.}  {\bf 244} (2013), 570--604.

\bibitem[LL18]{LL18}
H.~Li and B.~Liang. Mean dimension, mean rank, and von Neumann-L\"{u}ck rank. {\it J. Reine Angew. Math.} {\bf 739} (2018), 207--240. 


\bibitem[LL19]{LL19}
H.~Li and B.~Liang. Sofic mean length. {\it Adv. Math.} {\bf 353} (2019), 802–858.

\bibitem[LR21]{LR21}
H.~Li and Z.~Rong. Combinatorial independence and naive entropy. {\it Ergodic Theory Dynam. Systems} {\bf 41} (2021), no. 7, 2136--2147.

\bibitem[LL24]{LL24}
X.~Li and X.~Luo. Mean Hausdorff Dimension of Some Infinite Dimensional Fractals for Amenable Group Actions. {\it arXiv:2407.13489}.

\bibitem[L19]{L19}
B.~Liang. Dynamical correspondences of $L^2$-Betti numbers. {\it Groups Geom. Dyn.} {\bf 13} (2019), no. 1, 89--106. 

\bibitem[LSL24]{LSL24}
Y.~Liu, B.~Selmi, and Z.~Li. On the mean fractal dimensions of the Cartesian product sets. {\it Chaos Solitons Fractals} {\bf 180} (2024), Paper No. 114503, 9 pp.


\bibitem[L99]{L99}
 E.~Lindenstrauss. Mean dimension, small entropy factors and an embedding theorem. {\it Inst. Hautes Études Sci. Publ. Math.} {\bf 89} (1999), 227--262.



\bibitem[LT18]{LT18}
E.~Lindenstrauss and M.~Tsukamoto. From rate distortion theory to metric mean dimension: variational principle. {\it IEEE Trans. Inform. Theory } {\bf 64} (2018), no. 5, 3590--3609.

\bibitem[LT19]{LT19}
E.~Lindenstrauss and M.~Tsukamoto. Double variational principle for mean dimension. {\it Geom. Funct. Anal.} {\bf 29} (2019), no. 4, 1048--1109.

\bibitem[LW00]{LW00}
E. ~Lindenstrauss and B. ~Weiss. Mean topological dimension. {\it Israel J.\ Math.} {\bf 115} (2000), 1--24.


\bibitem[MT15]{MT15}
S.~Matsuo and M.~Tsukamoto. Brody curves and mean dimension. {\it J. Amer. Math. Soc.} {\bf 28} (2015), no. 1, 159--182.

\bibitem[N21]{N21}
Z.~Niu. $\mathcal{Z}$-stability of transformation group $C^\ast$-algebras. {\it Trans. Amer. Math. Soc.}  {\bf 374} (2021), 7525--7551.

\bibitem[Sch95]{Sch95}
K.~Schmidt. {\it Dynamical systems of algebraic origin}. Progress in Mathematics, 128. Birkh\"{a}user Verlag, Basel, 1995.


\bibitem[S19a]{S19a}
B.~Seward. Krieger's finite generator theorem for actions of countable groups I. {\it Invent. Math.} {\bf 215} (2019), no. 1, 265--310. 

\bibitem[S19b]{S19b}
B.~Seward. Krieger's finite generator theorem for actions of countable groups II. {\it J. Mod. Dyn.} {\bf 15} (2019), 1--39.

\bibitem[T18]{T18}
M.~Tsukamoto. Mean dimension of the dynamical system of Brody curves. {\it Invent. Math.} {\bf 211} (2018), no. 3, 935--968.

\bibitem[T19]{T19}
M.~Tsukamoto. Mean dimension of full shifts. {\it Israel J. Math.} {\bf 230} (2019), no. 1, 183--193.

\bibitem[T20]{T20}
M.~Tsukamoto. Double variational principle for mean dimension with potential. {\it Adv. Math.} {\bf 361} (2020), 106935, 53 pp.

\bibitem[T22]{T22}
M.~Tsukamoto. Mean Hausdorff dimension of some infinite dimensional fractals. {\it arXiv:2209.00512}.

\bibitem[YCZ22]{YCZ22}
R.~Yang, E.~Chen, and X.~Zhou. Bowen's equations for upper metric mean dimension with potential. {\it Nonlinearity}  {\bf 35} (2022), no. 9, 4905--4938.


\end{thebibliography}
\end{document}